\documentclass[11pt,A4paper]{article}

\usepackage[left=30mm,right=30mm,top=25mm,bottom=30mm]{geometry}
\usepackage{labelfig}
\usepackage{epsfig}
\usepackage{epstopdf}
\usepackage{soul}
\usepackage{xfrac}

\usepackage[percent]{overpic}
\usepackage{comment}
\usepackage{color}
\usepackage{amsthm,amsmath,amssymb}
\usepackage{booktabs}
\usepackage{mathpazo}
\usepackage{microtype}
\usepackage{overpic}
\usepackage{bm}
\usepackage{sectsty}
\usepackage[
	pdftitle={PDFTitle},
	pdfauthor={Alfredo Hubard},
	ocgcolorlinks,
	linkcolor=linkred,
	citecolor=linkred,
	urlcolor=linkblue]
{hyperref}

\definecolor{linkred}{RGB}{0,191,255} 
\definecolor{linkblue}{RGB}{16, 78, 139}

\usepackage[hang,flushmargin]{footmisc}
\usepackage{enumitem}
\usepackage{titlesec}
	\titlespacing{\section}{0pt}{12pt}{0pt}
	\titlespacing{\subsection}{0pt}{6pt}{0pt}
	
\titlelabel{\thetitle.\quad}

\makeatletter 

\long\def\@footnotetext#1{%
\H@@footnotetext{%
\ifHy@nesting 
\hyper@@anchor{\@currentHref}{#1}%
\else 
\Hy@raisedlink{\hyper@@anchor{\@currentHref}{\relax}}#1%
\fi 
}}

\def\@footnotemark{%
\leavevmode 
\ifhmode\edef\@x@sf{\the\spacefactor}\nobreak\fi 
\H@refstepcounter{Hfootnote}%
\hyper@makecurrent{Hfootnote}%
\hyper@linkstart{link}{\@currentHref}%
\@makefnmark 
\hyper@linkend 
\ifhmode\spacefactor\@x@sf\fi 
\relax 
}%

\makeatother 

\theoremstyle{plain}
\newtheorem{theorem}{Theorem}[section]
\newtheorem*{theorem-otal}{Theorem 1.3}
\newtheorem{proposition}[theorem]{Proposition}
\newtheorem{lemma}[theorem]{Lemma}

\theoremstyle{definition}

\newtheorem{remark}[theorem]{Remark}

\newcommand{\Z}{{\mathbb Z}}

\newcommand{\area}{{\rm area}}
\newcommand{\con}{{\rm con}}

\newcommand{\diam}{{\rm diam}}

\newcommand{\length}{{\rm length}}

\newcommand{\Cr}{{\rm cr}}

\sectionfont{\large \bfseries}
\subsectionfont{\normalsize}

\long\def\symbolfootnote[#1]#2{\begingroup%
\def\thefootnote{\fnsymbol{footnote}}\footnote[#1]{#2}\endgroup}

\def\blfootnote{\xdef\@thefnmark{}\@footnotetext}

\begin{document}

{\Large \bfseries Crossing numbers of dense graphs on surfaces}

{\large Alfredo Hubard, Arnaud de Mesmay and Hugo Parlier}\\

\vspace{0.5cm}

{\bf Abstract.}
In this paper, we provide upper and lower bounds on the crossing numbers of dense graphs on surfaces, which match up to constant factors.
First, we prove that if $G$ is a dense enough graph with $m$ edges and $\Sigma$ is a surface of genus $g$, then any drawing of $G$ on $\Sigma$ incurs at least 
$\Omega \left(\frac{m^2}{g} \log ^2 g\right)$ crossings. The poly-logarithmic factor in this lower bound is new even in the case of complete graphs and disproves a conjecture of Shahrokhi, Sz\'ekely and Vrt'o from 1996.
Then we prove a geometric converse to this lower bound: we provide an explicit family of hyperbolic surfaces such that for any graph $G$, sampling the vertices uniformly at random on this surface and connecting them with shortest paths yields $O\left(\frac{m^2}{g} \log ^2 g\right)$ crossings in expectation.

\section{Introduction} \label{s:introduction}

The relationship between graph drawing and surfaces has been a driving theme in combinatorics and topology over the last century. Crossing number inequalities for graph drawings in the plane, were first studied by Tur\'an~\cite{Turan1954}, Hill ~\cite{HararyHill1963} and Guy~\cite{Guy1960}. The topic has since led to a plethora of beautiful mathematics, including the so-called crossing number inequality, or crossing lemma~\cite{ACNS1982,Leighton1983}, and many generalizations. Combined with the Ringel-Youngs solution to the genus of complete graphs \cite{RingelYoungs1968}, the question of how to minimize the number of crossings when drawing complete graphs on surfaces arises. The precise question depends on the relationship between the number of vertices $n$ and the genus $g$ of the surface they are drawn on. In this paper we focus on a large number of vertices compared to the genus.

The Harary-Hill conjecture states that every drawing of the complete graph on $n$ vertices in the plane incurs at least $\frac{1}{4}\lfloor \frac{n}{2}\rfloor \lfloor \frac{n-1}{2}\rfloor\lfloor \frac{n-2}{2}\rfloor\lfloor \frac{n-3}{2}\rfloor\sim \frac{n^4}{64}$ crossings. Different constructions attain this bound, and the conjecture has been confirmed for several families of drawings (see~\cite{abrego2018bishellable,abrego2012two,abrego2014shellable,balko2015crossing}). As the genus grows however, the crossing number is naturally reduced, and the question becomes to quantify to which extent it does. In what follows, we denote by $\Cr_g(G)$ the crossing number of the graph $G$ when drawn in the closed surface of genus $g$.

\subsection{Lower bounds: a new crossing lemma for dense graphs on surfaces}

In \cite{ShahrokhiSzekelyVrto1998}, Shahrokhi, Sykora, Sz\'ekely and Vr'to generalized the crossing lemma to drawings on surfaces. They showed that for a graph with $n$ vertices and $m$ edges, if $\max(64g,8n) \leq m \leq \frac{n^2}{g}$ then $\Cr_g(G)\geq \Omega( \frac{m^3}{n^2})$. For larger values of $m$, this estimate does not hold anymore. Instead, they showed if $m\geq \max (64 g, \frac{n^2}{g},8n)$, then $\Cr_g(G)=\Omega(\frac{m^2}{g})$. They also proved the existence of a drawing of $K_n$ in a surface of genus $g$ with $O(\frac{n^4}{g} \log ^3 g)$ crossings. Shahrokhi, Sz\'ekely and Vr'to improved this bound in ~\cite{ShahrokhiSzekelyVrto1998} to 
$\Cr_g(K_n) =O(\frac{n^4}{g} \log^2 g)$. They conjectured that this construction could be improved, specifically, that there exists drawings satisfying $\Cr_g(K_n)=O (\frac{n^4}{g})$, matching their lower bound. Our main result disproves this conjecture, and shows that their upper bound has the correct order of growth. 

\begin{theorem}\label{main} Let $G$ be a graph with $n$ vertices and $m$ vertices and let $\Sigma$ be a surface of genus $g\geq 2$. For any $\varepsilon>0$, there exists a constant $B(\epsilon)$ such that if $m \geq 10^6 \max(n^{3/2}, ng,  \frac{n^2}{g^{1/2-\varepsilon}})$, then \[\Cr_g(G)\geq B(\varepsilon)m^2\frac{\log^2 g}{g}.\]
\end{theorem}

The poly-logarithmic factor in this theorem is new, even for complete graphs. For concreteness, in that case, the proof shows that we can safely take $B(\varepsilon)=10^{-4}$, but we have made no attempt to optimize the constants. At first glance, it is not obvious where the poly-logarithmic factor comes from (which explains why it was conjectured not to be there in the first place), but the proof is enlightening in this regard. As in \cite{hubard2025crossing}, it relies on hyperbolic geometry, and in particular the relationship between hyperbolic structures and embedded graphs via circle packings.

Lower bounds on the crossing number of complete graphs have multiple applications beyond graphs drawing. For instance, they were used by Biswal, Lee and Rao~\cite{biswal2010eigenvalue} and then Kelner, Lee, Price and Teng~\cite{kelner2011metric} to provide upper bounds on the eigenvalues of the Laplacians of planar and surface-embedded graphs. Plugging the bound of our Theorem~\ref{main} in~\cite[Theorem~3.1]{biswal2010eigenvalue} and in~\cite[Section~3]{kelner2011metric} immediately improves their bounds by polylogarithmic factors (but note that optimal bounds for this problem have already been obtained using different methods~\cite{amini2018transfer}).

\begin{remark}
We do not know what is the smallest $m$ for which our bound holds. Around $m=\frac{n^2}{2g}$, a graph consisting of a union of $g$ cliques of $\frac{n}{g}$ vertices each satisfies $m=g\binom{\frac{n}{g}}{2}\sim \frac{n^2}{2g}$ and has crossing number $O(g(\frac{n}{g})^4)=O(\frac{n^4}{g^3})$. This is of the same order as $O(\frac{m^3}{n^2})$ and as $O(\frac{m^2}{g})$, so the $\log^2 g$ factor is not yet apparent. Yet our theorem only applies when the number of edges is $\Omega(\frac{n^2}{g^{1/2-\varepsilon}})$. While the central case of Theorem \ref{main} is the complete graph, the main open question that is left open in this direction (beyond the challenge of finding the exact constants) is to understand better the different phase transitions that the crossing number undergoes as the number of edges grows relative to the number of vertices of the graph and the genus of the surface.
\end{remark}

\subsection{Upper bounds: geometric drawings in high genus}
Our proof of Theorem \ref{main} geometrizes a problem in combinatorial topology, and then uses geometric estimates. More specifically, the Cauchy-Schwarz inequality transforms lower bounds on the average edge length into lower bounds on the crossing number. We now discuss a variant in which the crossing number is defined in a geometric way from the get go.

The \emph{spherical crossing number} is analogous to the crossing number, but edges are drawn by spherical geodesic shortest paths on the round sphere. The Harary-Hill conjecture was recently \cite{streltsova2025spherical} confirmed for spherical crossing numbers. For a chosen surface and metric, the \emph{geometric random graph} on $n$ vertices and parameter $r$ is the intersection graph of metric balls of radius $r$ whose centers are chosen uniformly and independently at random. It comes naturally drawn with edges represented by shortest paths. Moon~\cite{Moon1965} showed that the geometric random drawing of the complete graph (with $r=\pi$) on the round sphere has $\frac{n^4}{64}+o(n^4)$ crossings in expectation, matching the conjectured optimal drawings at the first order. 

Let $G_{n,m,g}$ be a graph with $n$ vertices and $m$ edges that minimizes $\Cr_g$.
It was shown in \cite{pach1996new} that for every fixed $g$, in the regime $n\ll m \ll n^2$, the limit $\lim_{n \to \infty} \Cr_g(G_{n,m,g})\frac{n^2}{m^3}$  exists and is independent of $g$. This constant $c_{mid}$ is known as the \emph{midrange constant} and is known to be larger than $\frac{1}{27.7}\sim 0.03$ and smaller than $\frac{8}{9\pi^2}\sim~0.09$ \cite{bungener2024improving,czabarka2019some}. The upper bound construction is given by the geometric random graph on the round sphere.

Generalizing these results beyond the planar or constant genus case, one might wonder if optimal (or near optimal) drawings of complete graphs can be achieved by a drawing in which all edges are shortest paths on the surface enriched with a well chosen metric of constant curvature.

 The answer, at least for the projective plane with the round metric, is no: Elkies~\cite{elkies2017crossing} computed the expected crossing number of a complete random geometric graph on the round projective plane and conjectured it to be optimal. Later Arroyo, McQuillan, Richter, Salazar and Sullivan~\cite{arroyo2021drawings} constructed a topological drawing of the complete graph on the projective plane which has, asymptotically, less crossings than the geometric bound of Elkies. 
 
 Elkies also computed the expected crossing number of the random geometric drawing of the complete graph for every flat torus. He discovered that this expectation is minimized for the honeycomb torus, $\mathbb C/\Z[1,e^{2\pi i/3}]$ which turns out to also be extremal for the systolic inequality, among many other extremal problems in geometry.

 Recall that the systolic inequality for orientable surfaces of genus $g$ shows that for every Riemannian metric $m$, on a surface of genus $g$, if we denote by $sys(m)$ the length of the shortest non contractible curve, then $sys(m)^2\leq \area(m)\frac{\log^2 g}{g}$. The similarity of this formula with the estimate in Theorem~\ref{main} and the aforementioned result by Elkies makes one wonder if, among drawings in which edges are drawn by shortest paths in a hyperbolic surface, the drawing of the complete graph minimizing the number of crossings is realized on a surface that optimizes the systolic inequality. 
 
  Our second result follows this train of thought and provides the following converse to Theorem \ref{main}, that strengthens the aforementioned result of ~\cite{ShahrokhiSzekelyVrto1998}. 

\begin{theorem}\label{thm:upperbound}
    There exists a family of hyperbolic surfaces $S_k$ of genus $g_k \rightarrow \infty$ such that for any $n>0$, sampling $n$ points uniformly at random on $S_k$ and connecting them with shortest paths yields a drawing of $K_n$ on $S_k$ with $O(n^4 \frac{\log^2 g_k}{g_k})$ crossings in expectation.
\end{theorem}

Notice that by embedding a graph $G$ on $n$ vertices at random on the complete graph, this result implies that for any graph, there exists a geometric drawing with at most $O(m^2 \frac{\log^2 g_k}{g_k})$ crossings. 

The construction behind Theorem~\ref{thm:upperbound} also helps give geometric meaning to the quantities in Theorem~\ref{main}. Roughly speaking, this is how it works: we consider families of compact arithmetic surfaces, which are obtained by taking the quotient of the hyperbolic plane by specific well-chosen subgroups of triangle groups called congruence subgroups. Those have two very convenient properties. First, by construction, there is an isometry group acting transitively on their triangles. Second, the choice of congruence subgroups gives these surfaces expander-like properties, namely a uniform lower bound on their first positive eigenvalue, and hence they have diameter which is on the order of the logarithm of their genus $g_k$. Now by randomly placing vertices, and by taking shortest paths between them, each shortest path will cross $O(\log(g_k))$ triangles. As there are roughly $g_k$ triangles, this spaces out the intersections between shortest paths to get the desired bound. With this in mind, the logarithmic factor becomes a distance, and the genus in the denominator is the order of the isometry group or the area of the surface. We mention in passing that the compact arithmetic surfaces have been conjectured to be asymptotically optimal for the systolic inequality among hyperbolic surfaces~\cite{schaller1998geometry}.

\begin{remark}
    While the main contribution of Theorem~\ref{thm:upperbound} is its geometric nature, a topological drawing of $K_n$ which satisfies $\Cr_g(K_n)=O(n^4\frac{\log ^2 g}{g})$ can be obtained using a standard cut-flow duality approach~\cite{leighton1999multicommodity}: here is a quick sketch (the reader might compare it to the construction in ~\cite{ShahrokhiSzekelyVrto1998}). The idea is to paste pairs of pants following the combinatorics of $3$-regular expander graph on $g$ vertices.
It is well-known~\cite[Theorem~1]{leighton1999multicommodity}  that there exists uniform multi-commodity flows on expander graphs on $g$ vertices with vertex congestion $O(g\log g)$. If instead of sending one unit of flow per pair of vertices we send $\frac{n^2}{g^2}$ units of flow, then this turns into a vertex congestion of $O(\frac{n^2\log g}{g})$. Extracting an integer flow, we can interpret it as a map from the complete graph $K_n$ into the expander, which in turn can be readily transformed into a drawing of the complete graph on the surface. The crossing number of this drawing in each pair of pants is bounded above by the square of the congestion at the corresponding vertex. All in all, we obtain at most $O(g(\frac{n^2\log g}{g})^2)=O(n^4\frac{\log^2 g}{g})$ crossings, as desired. 
\end{remark}

\section{Preliminaries}

In this paper, all graphs are simple and loopless. A drawing of a graph $G$ on a closed surface $\Sigma$ is a continuous map from (the geometric realization of) $G$ to $\Sigma$. Said differently, we represent each vertex by a distinct point and each edge by an arc with endpoints on the points representing the vertices. Throughout the article, we are going to make some assumptions on drawings: No two vertices are represented by the same point, any edge is disjoint from any vertex, except when the vertex is one of its endpoints, and three distinct edges never pass through a same point (unless it is an endpoint of all three of them). In the following, we identify a vertex and the point that represents it and we identify an edge with the arc that represents it.

We denote by $\Sigma$ the closed orientable surface of genus $g$, which has Euler characteristic $\chi(\Sigma)=2-2g$. If $D$ is a surface, $g(D)$ denotes its genus. 

We are going to endow $\Sigma$ with a hyperbolic metric, that is a Riemannian metric of constant curvature $-1$. Its universal cover is the hyperbolic plane. We recall a couple of facts from Riemannian and hyperbolic geometry. If $D(c,r)$ is a disk in the hyperbolic plane, centered at $c$ of radius $r$, then 
\[\area(D(c,r))=4\pi \sinh^2 \frac{r}{2}\]

It is not hard to see from this inequality that for every $r$, \[\pi r^2 \leq  \area(D(c,r)),\]
meaning that hyperbolic disks are of greater area than their  Euclidean analogues.

 The Gauss-Bonnet theorem shows that a smooth Riemannian surface $\Sigma$ with smooth boundary satisfies 
 \[\int_\Sigma K(x) da(x)+\int_{\partial \Sigma} \kappa(x) d\ell(x)=2\pi \chi(\Sigma),\]

In the first integral, $K(x)$ denotes the Gaussian curvature of the surface $\Sigma$ at the point $x$, and $da$ is the Riemannian area element. In the second integral, $\kappa(x)$ denotes the geodesic curvature of the boundary of the surface $\Sigma$, and $d\ell$ is the Riemannian length element.  Finally $\chi(\Sigma)$ is the Euler characteristic of the surface. In a hyperbolic surface the Gaussian curvature is uniformly $K=-1$.

\section{The lower bound}

We start with a few simple geometric and combinatorial lemmas.

The following lemma is immediate by comparing the area formula of a hyperbolic disk and the area of a genus $g$ hyperbolic surface. 

\begin{lemma}\label{disk_embedded} The maximal radius of a disk embedded in a hyperbolic surface $\Sigma$ of genus $g$ is at most $\log g+\log 4$.
\end{lemma}
We note that an optimal upper bound on the size of a largest embedded disk was obtained by C. Bavard \cite{bavard}.

\begin{lemma}\label{disk_genus} If $\Sigma$ is a hyperbolic surface, for every $x\in \Sigma$ and $c>0$, the genus of the metric ball $B(x,c \log g)$ is at most $g^c/4+1$. 
\end{lemma}

\begin{figure}
    \centering
    \includegraphics[width=0.5\linewidth]{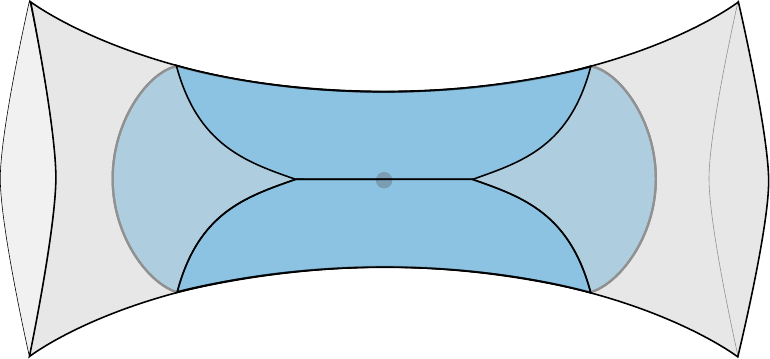}
    \caption{Even if the metric ball is not a disk, it can be parametrized as a self-touching disk with nonnegative geodesic curvature on its boundary. In this example, the center of the metric ball $x$ is behind the collar.}
    \label{MetricBall}
\end{figure}

\begin{proof}
The metric ball $B(x,c \log g)$ is in general not a disk and might self-intersect. If that is the case, we reparametrize it as an embedding $f\colon D \to B(x,c \log g)$ from an open disk $D$. It is not hard to see that one can assume that the closure of $f(D)$ is the closed ball  $B(x,c \log g)$ and the boundary of $f(D)$ is the cut locus of $x$.  This parametrization can be chosen so that the interior of $D$ is embedded and that the geodesic curvature of $\partial D$ is nonnegative everywhere, see Figure~\ref{MetricBall}.  The Gauss-Bonnet formula yields: 
\[-\area(D)=\int_D -1 da(x)= 2\pi \chi(\Sigma)-\int_{\partial D} \kappa(x) d\ell(x), \] hence 
\[\area(D)\geq -2\pi \chi(D)=4\pi g(D)-4\pi.\]
On the other hand, if we denote by $\tilde{D}(r)$ a disk in the hyperbolic plane of the same radius as $D$, we obtain 
\[\area(D)\leq \area(\tilde{D}) = 4 \pi \sinh^2(c \log g/2) \leq \pi e^{c \log g}= \pi g^c.\]
Putting together these the two inequalities, we obtain $g(D)\leq g^c/4+1$, as claimed.
\end{proof}

The following lemma allows us to find dense low-diameter subgraphs in any dense graph.

\begin{lemma}\label{lowdiameter}
Let $G$ be a graph with $m$ edges and $n$ vertices. There exists a vertex $x$ such that the induced subgraph of vertices at graph distance at most two from $x$ has at least $m^2/8n^2$ edges. 
\end{lemma}

\begin{proof}
Starting with $G$, erase a vertex with degree less than $\frac{m}{2n}$, and repeat until every vertex has degree at least $\frac{m}{2n}$. This erases at most $n\frac{m}{2n}=\frac{m}{2}$ edges and the remaining graph is thus non-empty. Now, pick $x$ to be an arbitrary non-isolated vertex on the remaining subgraph. There are at least $\frac{1}{8}\frac{m^2}{n^2}$ edges with at least one vertex that neighbors $x$, which implies the lemma.
\end{proof}

\begin{lemma}\label{conclusion} Let $G$ a graph with $n$ vertices and $m$ edges, drawn on a surface $\Sigma$ with a metric such that the average length of an edge is less than $\ell$. Then there exists $x \in \Sigma$ such that the metric ball $B(x,4\ell)\subset \Sigma$ contains at least $\frac{1}{32}\frac{m^2}{n^2}$ edges.
\end{lemma}

\begin{proof} Let $G'$ be the subgraph of edges of length at most $2\ell$. By Markov's inequality, it has at least $\frac{m}{2}$ edges. Now apply Lemma~\ref{lowdiameter} to this graph. We obtain a metric ball that contains at least $\frac{1}{32}\frac{m^2}{n^2}$ edges of length at most $2\ell$, as claimed.
\end{proof}

From~\cite[Theorem~3.3]{ShahrokhiSykoraSzekelyVrto1996-2} we will use the following result, where we have chosen $c=2 \sqrt{3}$ and $\epsilon=3/4$ in their proof.

\begin{lemma}\label{their_theorem} For every graph $G$, if $m\geq \max(2 \sqrt{3} n,12 g, 24 \frac{n^2}{g})$, then $\Cr_g(G)\geq \frac{m^2}{32 g}$. 
\end{lemma}

We now have all the tools to prove our main theorem.

\begin{proof}[Proof of Theorem \ref{main}]
 
Given a drawing of $G$ on $\Sigma$ with a minimal number of crossings, let $H$ be the graph embedded in $\Sigma$ in which each crossing is replaced by a vertex of degree $4$. We add edges arbitrarily between the vertices of $H$ to obtain a triangulation $T$ of $\Sigma$. Note that this might require multiple edges but these will be discarded later on.

Applying the Koebe-Andreev-Thurston circle packing theorem (see for example Stephenson~\cite[Theorem~4.3]{stephenson2005introduction}) to $T$, we obtain a unique hyperbolic surface $X$, homeomorphic to $\Sigma$, together with a family of disks on $X$ which are in bijection with the vertices of $H$. The interiors of any two disks are disjoint, and two boundary circles touch if and only if the corresponding vertices share an edge in $T$. Drawing the vertices of $T$ at the centers of the circles and the edges of $T$ by shortest paths between the centers of touching circles yields an embedding of $T$ on $\Sigma$. 

We can draw the original graph $G$ on $X$, representing every edge of $G$ by a concatenation of shortest paths between the centers of the circles. Notice that when we do this, no edge of $H$ is used by more than one edge in $G$.

We can assume that $\varepsilon <1/4$, since this implies the theorem for higher values of $\varepsilon$. The proof now splits into two cases.

\textbf{Case 1:} If the average length of an edge of $G$ in $X$ is less than $\frac{1}{4}(1-2\varepsilon) \log g$, then, by Lemma \ref{conclusion}, there exist a metric ball $B:=B(x, (1-2\varepsilon) \log g)$, with $x$ one of the points representing the vertices of $G$, such that the subgraph $G'$ of edges of $G$ contained in $B$ has at least $m'\geq \frac{1}{32}\frac{m^2}{n^2}$ edges. By Lemma \ref{disk_genus}, $B(x, (1-2\varepsilon) \log g)$ has genus at most $g':=g^{1-2\varepsilon}+1$. We can assume, without loss of generality, that it has exactly this maximal genus $g'$, as any smaller value would make our conclusion even stronger. To apply Lemma~\ref{their_theorem}, we verify that our assumptions imply $$m'\geq \frac{1}{32}\frac{m^2}{n^2}\geq \max\left\{2 \sqrt{3} n,12 (g^{1-2\varepsilon}+1), 24 \frac{n^2}{g^{1-2\varepsilon}+1}\right\}.$$
So, if the average length is at most $\frac{1}{4}\left(1-2\varepsilon\right) \log g$, then $\Cr(G)\geq \frac{m^2}{32\left(g^{1-2\varepsilon}+1\right)}$ which is greater than $B(\varepsilon) m^2\frac{\log^2 g}{g}$ for any $g$ and small enough $B(\varepsilon)=O(\varepsilon^2)$.

\textbf{Case 2:} If the average length of an edge of $G$ is $X$ is at least $\frac{1}{4}(1-2\varepsilon) \log g$, then it is at least $\frac{1}{8} \log g$ since $\varepsilon<1/4$.

Consider  $\left(r_1,r_2, \ldots r_{|V(H)\setminus V(G)|}\right)$, the vector of hyperbolic radii of the circles that correspond to crossings in $H$. We now apply the Cauchy-Schwarz inequality to its scalar product with the vector $1_{V(H)\setminus V(G)}$:

\[\left(\sum_{i \in V(H)\setminus V(G)} r_i\right)^2 \leq \left(\sum_{i \in {V(H)\setminus V(G)}} r_i^2\right) \,\left|V(H) \setminus V(G)\right| =\left(\sum_{i \in {V(H)\setminus V(G)}} r_i^2\right) \, \Cr_g(G)\]

Euclidean disks are of lesser area than their hyperbolic counterparts. And since the disks do not overlap, this implies that \[\pi \sum_{i \in V(H)} r_i^2 \leq \sum \area(D(c_i,r_i))\leq \area\left(X\right) < 2\pi\left(2g-2\right).\]
Therefore 
\[\Cr_g(G)\geq \frac{1}{4g}\left(\sum_{i \in V(H)\setminus V(G)} r_i\right)^2.\]

Consider the set $$I:=\left\{i \in V(G): r_i\geq\frac{1}{4}\right\},$$ 
and the truncated surface

 $$X':=X \setminus\cup_{i\in I} D(i,r_i).$$

Since for $i$ in $I$ we have $\area(D(i,r_i))\geq \frac{\pi}{16}$, and these disks are disjoint, comparing with the area of the surface yields $|I|\leq 64g$.  Define $E'(G)$ to be the edges of $G$ in $X$ that are longer than $1$, and are either not incident to any of the vertices in $I$, or are longer than $4 \log g+ 4 \log 4$. By Lemma~\ref{disk_embedded} and our choice of parameters, at least half of the length of edges in $E'$ is in $X'$: for every $e\in E'$,
\[\length_X(e)\leq 2\length_{X'}\,(e \cap X').\]

 Since $E\setminus E'$ is the union of all the edges of length at most $1$ and of all the edges of length at most $4 \log g+4 \log 4$ incident to a vertex in $I$, we have
$$\sum_{e \in E\setminus E'} \length_{X'}(e \cap X')\leq \sum_{e \in E\setminus E'} \length_X(e)\leq
n(64g)(4\log g +4 \log 4)+m\leq 
346 ng\log g+m.$$
Putting everything together we obtain, denoting by $d(i)$ the degree of a vertex $i$ of $V(G)$:
\begin{eqnarray*}
    4\sum_{i \in V(H)\setminus V(G)} r_i &=&\sum_{e \in E}\length_{X'}(e\cap X')-\sum_{i\in V(G) \setminus I}r_id(i)\\
&\geq& \frac{1}{2}\sum_{e \in E'}\length_{X}(e) -\frac{2m}{4}\\
&\geq& \frac{1}{2}\left(\sum_{e \in E}\length_{X}(e)-m-346\,ng\log g\right)-\frac{m}{2}\\ 
  & \geq & \frac{m}{16} \log g-m-173\,ng\log g
\geq \frac{m}{32} \log g
\end{eqnarray*}

for large enough $\log g$, using our assumptions on $m$.

We conclude that $$\Cr(G) \geq \frac{1}{2^{14}}\frac{m^2}{g}\log^2 g,$$ 
whenever $\log g$ is larger than some universal constant. For smaller values of $g$, we conclude using Lemma \ref{their_theorem} again (but note that the value of $B(\varepsilon)$ depends on $\varepsilon$ as provided by the first case).

\end{proof}

\section{The upper bound}

The family of surfaces in Theorem~\ref{thm:upperbound} is a family of arithmetic surfaces. Such families of surfaces are defined via number-theoretical constructions and have a host of remarkable properties akin to those of expander graphs. Our Theorem~\ref{thm:upperbound} holds for \emph{any} family of compact arithmetic surfaces obtained by taking congruence covers but for the sake of concreteness we exhibit a specific one in the proof which is well tailored for our purposes. Its relevant properties are summarized in the following proposition.

\begin{proposition}\label{triangle444}
  There exists an infinite family of hyperbolic surfaces $S_k$ with the following properties:
  \begin{enumerate}
      \item\label{triangles} For every $k$, the surface $S_k$ admits a geodesic triangulation $\mathcal{T}_k$ made of identical hyperbolic triangles of angles $(\pi/4,\pi/4,\pi/4)$, and thus of area $\pi/4$.
      \item The genus $g_k$ of the surfaces $S_k$ is increasing and goes to infinity.
      \item\label{trans} For every $k$, there is a group $G_k$ acting on $S_k$ whose action is area-preserving and which acts transitively on the set of triangles in $\mathcal{T}_k$.
      \item\label{diameter} For every $k$, the diameter of the surface $S_k$ is $O(\log g_k)$.     
  \end{enumerate}
\end{proposition}

The details of the construction and the proofs of the properties rely on deep number-theoretical arguments which fall outside of the scope of this paper. In the proof below we provide an overview of the construction and pointers to references providing the complete details of its properties.

\begin{proof}
    Let $T$ denote a hyperbolic triangle in the hyperbolic plane $\mathbb{H}$ with angles $(\pi/4,\pi/4,\pi/4)$ and denote its edges by $a,b$ and $c$. Such a triangle is unique up to isometries of the hyperbolic plane. The \emph{triangle group} $T(4,4,4)$ is the group generated by reflections with respect to $a$, $b$ and $c$. The elements in $T(4,4,4)$ are isometries of the hyperbolic plane but might be orientation-reversing. 

    It follows from a classification theorem of Takeuchi~\cite{takeuchi1977arithmetic} that the group $T(4,4,4)$ is \emph{arithmetic}. The precise definition of this notion is technical and outside the scope of this paper and we refer to standard textbooks such as Katok~\cite[Chapter~5]{katok1992fuchsian} or Maclachlan and Reid~\cite{maclachlan2003arithmetic}, but this implies that there is an infinite family of subgroups $\Gamma_k  \subseteq T(4,4,4)$ of finite index with the properties that (i) the elements of $\Gamma_k$ are orientation-preserving isometries of $\mathbb{H}$, (ii) for each $k$, the quotient $\Gamma_k / \mathbb{H}$ is compact and (iii) the resulting family of surfaces $S_k$ has genus $g_k$ going to infinity. Since $\Gamma_k$ is a subgroup of $T(4,4,4)$, the tessellation of $\mathbb{H}$ by the geodesic triangles of angles $(\pi/4,\pi/4,\pi/4)$ projects to a geodesic triangulation $\mathcal{T}_k$ of each $S_k$. Since the group $T(4,4,4)$ acts transitively on this tessellation of $\mathbb{H}$, the quotient group $G_k:=T(4,4,4)/\Gamma_k$ acts transitively on the triangles of $\mathcal{T}_k$. As the elements of $T(4,4,4)$ are isometries, this action is area-preserving. 

    Finally, for any hyperbolic surface $S$, its diameter $\diam(S)$is related to its \emph{spectral gap} $\lambda_1(S)$ which is the value of the first non-zero eigenvalue of its Laplace-Beltrami operator. We refer to Buser's classic textbook for an introduction to the spectral geometry of hyperbolic surfaces~\cite{buser2010geometry}. Following~\cite{magee2020letter}, we first have~\cite{buser1982note} $\lambda_1(S)\leq 2h(S)+10h^2(S)$, where $h(S):= \min_{M}  \frac{|\partial M|}{\min(\area(M),\area(S \setminus M))}$ is the \emph{Cheeger constant} of $S$. On the other hand, the diameter of a hyperbolic surface satisfies~\cite[Paragraph~4.6]{mirzakhani2013growth}

    \[ \diam(S) \leq 2\left(r_0+\frac{1}{h(S)}\log\left(\frac{\area(S)}{2|B(r_0)|}\right)\right),\]
    for any $r_0$ such that every disk $D(c,r_0)$ is embedded, and where $|B(r_0)|=4\pi \sinh^2(r_0/2)$ is its area. For our surfaces $S_k$, one can take $r_0$ to be a constant, and since $\area(S_k)=\Theta(g_k)$, we conclude that if $\lambda_1(S_k)$ is bounded by a constant that does not depend on $k$, we obtain the diameter bound $\diam(S_k)=O(\log g_k)$.
    
    It turns out that for all arithmetic surfaces, there is a universal lower bound $\lambda_1(S_k)\geq 3/16$. This follows from the Selberg 3/16 theorem~\cite{selberg1965estimation} (which proves this for specific families of \emph{non-compact} arithmetic surfaces) coupled with the Jacquet-Langlands transfer principle~\cite{jacquet1970automorphic}, as explained for example by Buser~\cite{buser1984bipartition}, Lubotzky~\cite[Theorem~6.3.4]{lubotzky1994discrete} and Vignéras~\cite{vigneras1983quelques}. For the specific case of the $T(4,4,4)$ group at hand, this is precisely stated in Lubotzky~\cite[Theorem~4.4.3 and the pararaph before]{lubotzky1994discrete}.
\end{proof}

The following lemma allows us to connect the geometric properties of shortest paths on $S_k$ with their combinatorial intersections with $\mathcal{T}_k$.

\begin{lemma}\label{lem:combinatorics}
    Let $\gamma$ be a shortest path of length $L$ on $S_k$ between two points. Then $\gamma$ intersects $O(L)$ triangles of $\mathcal{T}_k$, counted with multiplicity.
\end{lemma}

\begin{proof}
    Since triangles in $\mathcal{T}_k$ have angles $(\pi/4,\pi/4,\pi/4)$, there are exactly eight triangles incident to each vertex. We pick a constant $\varepsilon>0$ that is smaller than half of the side-length of a triangles in $\mathcal{T}_k$ (for example, hyperbolic trigonometry shows that $\varepsilon =1/2$ will do), and consider, for each vertex $v$ of $\mathcal{T}_k$, the ball $B(v,\varepsilon)$.
    These balls cut each triangle $T=(v_1,v_2,v_3)$ into its \emph{thick part} $T \setminus (B(v_1,\varepsilon) \cup B(v_2,\varepsilon) \cup B(v_3,\varepsilon)$ and its \emph{thin part} $\bigcup_i T \cap B(v_i,\varepsilon)$. Our choice of $\varepsilon$ ensures that the thick part is non-empty and that the thin part of each triangle is made of three connected components.
    If a connected component $\gamma \cap T$ meets the thick part of a triangle $T$, then it has length at least $c$ where $c$ is a constant depending only on $\varepsilon$ (if we chose $\varepsilon=1/2$, some more hyperbolic trigonometry shows that we can take $c=1/2$). On the other hand, if it only meets the thin part of $T$, then it meets the thin part of at most four triangles before meeting a thick part again, since a shortest path in $B(v,\varepsilon)$ meets at most four triangles. Therefore $\gamma$ meets at most $O(L)$ triangles, counted with multiplicity.
    \end{proof}

We now have all the tools to prove Theorem~\ref{thm:upperbound}.

\begin{proof}[Proof of Theorem~\ref{thm:upperbound}]
    For each edge $e$ of a drawing of $K_n$ on $S_k$, we cut $e$ into maximal geodesic segments $(e_1,\ldots , e_m)$ where each $e_i$ is fully contained in a hyperbolic triangle $T \in \mathcal{T}_k$. Now, for each triangle $T$ in $\mathcal{T}_k$, we define its \emph{congestion} $\con(T)$ to be the number of geodesics segments it contains, summing over all the edges. The \emph{congestion} $\con(K_n)$ of a drawing of $K_n$ is the maximum congestion over all triangles. Since shortest paths pairwise cross at most once, the crossing number of $K_n$ is upper bounded by $\con^2(K_n) |\mathcal{T}_k|$.

By Proposition~\ref{triangle444}~\ref{triangles}, since the area of $S_k$ is $\Theta(g_k)$, there are $\Theta(g_k)$ triangles in $\mathcal{T}_k$. By Proposition~\ref{triangle444}~\ref{diameter}, since edges are shortest paths, each edge in the drawing of $K_n$ on $S_k$ has length $O(\log g_k)$. Coupled with Lemma~\ref{lem:combinatorics}, this implies that each edge is cut into $O(\log g_k)$ maximal geometric segments. Therefore, \begin{align}\sum_{T \in \mathcal{T}_k} \con(T)=O(n^2 \log g_k). \label{eq1}\end{align}

    Now, by Proposition~\ref{triangle444}~\ref{trans}, the action of $G_k$ on $S_k$ is area-preserving and thus for any $h \in G_k$, the drawings $K_n$ and $h(K_n)$ are equiprobable. Since the action of $G_k$ is transitive on the triangles of $\mathcal{T}_k$, it follows that each triangle has exactly the same congestion in expectation, which therefore equals $\mathbb{E}(\con(K_n))$. It follows from Equation~\ref{eq1} that $\mathbb{E}(\con(K_n)) =O(n^2 \frac{\log g_k}{g_k})$, and thus that $\mathbb{E}(\Cr(K_n))=O(n^4 \frac{\log^2 g_k}{g_k})$. This concludes the proof.
\end{proof}

Here again, we have made no efforts to optimize the constants, tracking them in this proof shows that the constant hidden in the $O(\cdot)$ notation is smaller than $10^5$.

\bibliographystyle{plain} 
\bibliography{crossing}

\begin{thebibliography}{10}

\bibitem{abrego2018bishellable}
Bernardo~M {\'A}brego, Oswin Aichholzer, Silvia Fern{\'a}ndez-Merchant, Dan
  McQuillan, Bojan Mohar, Petra Mutzel, Pedro Ramos, R~Bruce Richter, and
  Birgit Vogtenhuber.
\newblock Bishellable drawings of {$K_n$}.
\newblock {\em SIAM Journal on Discrete Mathematics}, 32(4):2482--2492, 2018.

\bibitem{abrego2012two}
Bernardo~M {\'A}brego, Oswin Aichholzer, Silvia Fern{\'a}ndez-Merchant, Pedro
  Ramos, and Gelasio Salazar.
\newblock The 2-page crossing number of {$K_n$}.
\newblock {\em Discrete \& Computational Geometry}, 49:747--777, 2013.

\bibitem{abrego2014shellable}
Bernardo~M. {\'A}brego, Oswin Aichholzer, Silvia Fernández-Merchant, Pedro
  Ramos, and Gelasio Salazar.
\newblock Shellable drawings and the cylindrical crossing number of {$K_n$}.
\newblock {\em Discrete and Computational Geometry}, 52(4):743--753, 2014.

\bibitem{ACNS1982}
Mikl{\'o}s Ajtai, V{\'a}clav Chv{\'a}tal, Monroe Newborn, and Endre
  Szemer{\'e}di.
\newblock Crossing-free subgraphs.
\newblock In {\em Theory and Practice of Combinatorics}, pages 9--12.
  North-Holland Mathematics Studies, 1982.

\bibitem{amini2018transfer}
Omid Amini and David Cohen-Steiner.
\newblock A transfer principle and applications to eigenvalue estimates for
  graphs.
\newblock {\em Commentarii Mathematici Helvetici}, 93(1):203--223, 2018.

\bibitem{arroyo2021drawings}
Alan Arroyo, Dan McQuillan, R~Bruce Richter, Gelasio Salazar, and Matthew
  Sullivan.
\newblock Drawings of complete graphs in the projective plane.
\newblock {\em Journal of Graph Theory}, 97(3):426--440, 2021.

\bibitem{balko2015crossing}
Martin Balko, Radoslav Fulek, and Jan Kynčl.
\newblock Crossing numbers and combinatorial characterization of monotone
  drawings of {$K_n$}.
\newblock {\em Discrete and Computational Geometry}, 53(1):107--143, 2015.

\bibitem{bavard}
Christophe Bavard.
\newblock Disques extr\'{e}maux et surfaces modulaires.
\newblock {\em Ann. Fac. Sci. Toulouse Math. (6)}, 5(2):191--202, 1996.

\bibitem{biswal2010eigenvalue}
Punyashloka Biswal, James~R Lee, and Satish Rao.
\newblock Eigenvalue bounds, spectral partitioning, and metrical deformations
  via flows.
\newblock {\em Journal of the ACM (JACM)}, 57(3):1--23, 2010.

\bibitem{bungener2024improving}
Aaron B{\"u}ngener and Michael Kaufmann.
\newblock Improving the crossing lemma by characterizing dense 2-planar and
  3-planar graphs.
\newblock In {\em 32nd International Symposium on Graph Drawing and Network
  Visualization (GD 2024)}, pages 29--1. Schloss Dagstuhl--Leibniz-Zentrum
  f{\"u}r Informatik, 2024.

\bibitem{buser1982note}
Peter Buser.
\newblock A note on the isoperimetric constant.
\newblock In {\em Annales scientifiques de l'{\'E}cole normale sup{\'e}rieure},
  volume~15, pages 213--230, 1982.

\bibitem{buser1984bipartition}
Peter Buser.
\newblock On the bipartition of graphs.
\newblock {\em Discrete applied mathematics}, 9(1):105--109, 1984.

\bibitem{buser2010geometry}
Peter Buser.
\newblock {\em Geometry and spectra of compact Riemann surfaces}.
\newblock Springer Science \& Business Media, 2010.

\bibitem{czabarka2019some}
Éva Czabarka, Ignatius Singgih, László~A. Székely, and Zhiyu Wang.
\newblock Some remarks on the midrange crossing constant.
\newblock {\em Acta Mathematica Hungarica}, 157(2):187--203, 2019.

\bibitem{elkies2017crossing}
Noam~D Elkies.
\newblock Crossing numbers of complete graphs.
\newblock {\em The Mathematics of Various Entertaining Subjects: Research in
  Games, Graphs, Counting, and Complexity, Volume 2}, 2:218, 2017.

\bibitem{Guy1960}
Richard~K. Guy.
\newblock The crossing number of the complete graph.
\newblock {\em Journal of Combinatorial Theory}, 4:376--390, 1960.

\bibitem{HararyHill1963}
Frank Harary and Anthony Hill.
\newblock On the number of crossings in a complete graph.
\newblock {\em Proceedings of the Edinburgh Mathematical Society},
  13(4):333--338, 1963.

\bibitem{hubard2025crossing}
Alfredo Hubard and Hugo Parlier.
\newblock Crossing number inequalities for curves on surfaces.
\newblock {\em arXiv preprint arXiv:2504.00916}, 2025.

\bibitem{jacquet1970automorphic}
H~Jacquet and RP~Langlands.
\newblock Automorphic forms on {$GL(2)$}.
\newblock {\em Lecture Notes in Mathematics}, 114, 1970.

\bibitem{katok1992fuchsian}
Svetlana Katok.
\newblock {\em Fuchsian groups}.
\newblock University of Chicago press, 1992.

\bibitem{kelner2011metric}
Jonathan~A Kelner, James~R Lee, Gregory~N Price, and Shang-Hua Teng.
\newblock Metric uniformization and spectral bounds for graphs.
\newblock {\em Geometric and Functional Analysis}, 21(5):1117--1143, 2011.

\bibitem{Leighton1983}
Frank~Thomson Leighton.
\newblock {\em Complexity Issues in {VLSI}: Optimal Layouts for the
  Shuffle-Exchange Graph and Other Networks}.
\newblock MIT Press, Cambridge, MA, 1983.

\bibitem{leighton1999multicommodity}
Tom Leighton and Satish Rao.
\newblock Multicommodity max-flow min-cut theorems and their use in designing
  approximation algorithms.
\newblock {\em Journal of the ACM (JACM)}, 46(6):787--832, 1999.

\bibitem{lubotzky1994discrete}
Alex Lubotzky.
\newblock {\em Discrete groups, expanding graphs and invariant measures},
  volume 125.
\newblock Springer Science \& Business Media, 1994.

\bibitem{maclachlan2003arithmetic}
Colin Maclachlan and Alan~W Reid.
\newblock {\em The Arithmetic ofHyperbolic 3-Manifolds}.
\newblock Springer, 2003.

\bibitem{magee2020letter}
Michael Magee.
\newblock Letter to {B}ram {P}etri, 2020.
\newblock Available at \url{https://www.mmagee.net/diameter.pdf}.

\bibitem{mirzakhani2013growth}
Maryam Mirzakhani.
\newblock Growth of {W}eil-{P}etersson volumes and random hyperbolic surface of
  large genus.
\newblock {\em Journal of Differential Geometry}, 94(2):267--300, 2013.

\bibitem{Moon1965}
John~W. Moon.
\newblock On the distribution of crossings in random complete graphs.
\newblock {\em Journal of the Society for Industrial and Applied Mathematics},
  13(2):506--510, 1965.

\bibitem{pach1996new}
János Pach, Gábor Tóth, and Joel Spencer.
\newblock New bounds on crossing numbers.
\newblock {\em Discrete \& Computational Geometry}, 16(3):291--312, 1996.

\bibitem{RingelYoungs1968}
Gerhard Ringel and J.~W.~T. Youngs.
\newblock Solution of the {Heawood} map-coloring problem.
\newblock {\em Proceedings of the National Academy of Sciences},
  60(2):438--445, 1968.

\bibitem{schaller1998geometry}
Paul Schmutz~Schaller.
\newblock Geometry of {R}iemann surfaces based on closed geodesics.
\newblock {\em Bulletin of the American Mathematical Society}, 35(3):193--214,
  1998.

\bibitem{selberg1965estimation}
Atle Selberg.
\newblock On the estimation of {F}ourier coefficients of modular forms.
\newblock In {\em Proceedings of Symposia in Pure Mathematics}, pages 1--15.
  American Mathematical Society, 1965.

\bibitem{ShahrokhiSykoraSzekelyVrto1996-2}
Farhad Shahrokhi, Ondrej S{\'y}kora, L{\'a}szl{\'o}~A. Sz{\'e}kely, and Imrich
  Vrt'o.
\newblock Crossing numbers: Bounds and applications.
\newblock In {\em Intuitive Geometry}, volume~6 of {\em Bolyai Society
  Mathematical Studies}, pages 179--206, 1996.

\bibitem{ShahrokhiSzekelyVrto1998}
Farhad Shahrokhi, L{\'a}szl{\'o}~A. Sz{\'e}kely, and Imrich Vrt'o.
\newblock Drawing of graphs on surfaces with few crossings.
\newblock {\em Algorithmica}, 16(1):118--131, 1996.

\bibitem{stephenson2005introduction}
Kenneth Stephenson.
\newblock {\em Introduction to circle packing: The theory of discrete analytic
  functions}.
\newblock Cambridge University Press, 2005.

\bibitem{streltsova2025spherical}
Elizaveta Streltsova and Uli Wagner.
\newblock Sublevels in arrangements and the spherical arc crossing number of
  complete graphs.
\newblock {\em arXiv preprint arXiv:2504.07770}, 2025.

\bibitem{takeuchi1977arithmetic}
Kisao Takeuchi.
\newblock Arithmetic triangle groups.
\newblock {\em Journal of the Mathematical Society of Japan}, 29(1):91--106,
  1977.

\bibitem{Turan1954}
P{\'a}l Tur{\'a}n.
\newblock A note of welcome.
\newblock {\em Journal of Graph Theory}, 1(1):7--9, 1954.

\bibitem{vigneras1983quelques}
Marie-France Vign{\'e}ras.
\newblock Quelques remarques sur la conjecture $\lambda_1\geq 1/4$, seminar on
  number theory, paris 1981--82.
\newblock {\em Progr. Math}, 38:321--343, 1983.

\end{thebibliography}

{\it Addresses:}\\
Department of Mathematics, University of Fribourg, Switzerland\\
Université Gustave Eiffel, CNRS, LIGM, Marne-la-Vallée, France.\\
{\it Emails:}\\
hugo.parlier@unifr.ch\\
alfredo.hubard@univ-eiffel.fr
arnaud.de-mesmay@univ-eiffel.fr

\end{document}